\documentclass[10pt]{amsart}
\usepackage{amsthm, amsfonts, amssymb, amsmath, graphicx, enumitem, bbold, caption}
\usepackage{lmodern}
\usepackage[margin=4cm]{geometry}

\newcommand{\F}{\mathbb{F}}

\newcommand{\e}{\varepsilon}

\newcommand{\la}{\langle}
\newcommand{\ra}{\rangle}

\newcommand{\ct}{\mathbb{1}}

\newtheorem{thm}{Theorem}[section]

\newtheorem{lem}[thm]{Lemma}

\theoremstyle{definition}

\begin{document}
\title{A relative bound for independence} 
\author{Bogdan Nica}
\begin{abstract}
We prove an upper bound for the independence number of a graph in terms of the largest Laplacian eigenvalue, and of a certain induced subgraph. Our bound is a refinement of a well-known Hoffman-type bound.
\end{abstract}

\address{\newline Department of Mathematics and Statistics \newline McGill University, Montreal}
\date{\today}

\maketitle
\section{Introduction}
A number of powerful bounds for combinatorial graph invariants rely on spectral information, and much has been written about this rich stream in spectral graph theory. Herein, we focus on eigenvalue bounds for the independence number of a graph. 

Let $X$ be a non-empty graph on $n$ vertices. The independence number of $X$ is denoted, as usual, by $\alpha$. The best-known spectral estimate for the independence number is the \emph{Hoffman bound}: if $X$ is regular of degree $d$, then 
\begin{align}\label{hoff-ad}
\alpha\leq n\:\frac{-\theta_{\min}}{d-\theta_{\min}}
\end{align}
where $\theta_{\min}$ is the smallest adjacency eigenvalue. The Hoffman bound has been extended by Haemers \cite{Ham} to graphs which are not necessarily regular, as follows:
\begin{align*}
\alpha\leq n\:\frac{-\theta_{\min}\:\theta_{\max}}{\delta^{2}-\theta_{\min}\:\theta_{\max}}
\end{align*}
where $\theta_{\min}$ and $\theta_{\max}$ are the extremal adjacency eigenvalues, and $\delta$ denotes the minimal degree. If $X$ is $d$-regular, then $\delta=d$ and $\theta_{\max}=d$, so we recover the bound \eqref{hoff-ad}.

Our interest in this paper lies on the Laplacian side. The two spectral perspectives, adjacency and Laplacian, are essentially equivalent on regular graphs; on irregular graphs, they are genuinely different. The Laplacian formulation of the Hoffman bound \eqref{hoff-ad} reads as follows: if $X$ is regular of degree $d$, then 
\begin{align}\label{hoff-lap}
\alpha\leq n\Big(1-\frac{d}{\lambda_{\max}}\Big)
\end{align}
where $\lambda_{\max}$ is the largest Laplacian eigenvalue. Once again, the Hoffman bound \eqref{hoff-lap} can be generalized to graphs which are not necessarily regular. It turns out that one simply has to use the minimal degree $\delta$.

\begin{thm}\label{thm: Hlap}
The independence number of $X$ satisfies
\begin{align*}
\alpha\leq n\Big(1-\frac{\delta}{\lambda_{\max}}\Big).
\end{align*}
\end{thm}

It appears that this result has been rediscovered several times (van Dam - Haemers \cite[Lem.3.1]{vDH}, Zhang \cite[Cor.3.3]{Zh}, Godsil - Newman \cite[Cor.3.6]{GN}, Lu - Liu - Tian \cite[Thm.3.2]{LLT}). We refer to the estimate of Theorem~\ref{thm: Hlap} as the \emph{Hoffman-type bound}. Clearly, it turns into the Hoffman bound \eqref{hoff-lap} in the case of regular graphs.

The weakness of the Hoffman-type bound comes from its undue dependence on the minimal degree. Consider, for example, the addition of a pendant vertex to a given graph; then the independence number, the size, and the largest Laplacian eigenvalue remain essentially the same, but the minimal degree can change drastically. This motivates us to look for a more stable refinement of the Hoffman-type bound.

We prove a refinement which depends on the independence number of a certain subgraph. Given a graph $X$, the \emph{derived graph} $X'$ is the subgraph obtained by deleting the vertices of $X$ which have maximal degree. Henceforth, the maximal degree is denoted by $\Delta$; let us recall that the minimal degree is denoted by $\delta$.

\begin{thm}\label{thm: rel}
The independence number of $X$ satisfies
\begin{align*}
\alpha\leq n\Big(1-\frac{\Delta}{\lambda_{\max}}\Big)+\alpha'\frac{\Delta-\delta}{\lambda_{\max}-\delta}
\end{align*}
where $\alpha'$ is the independence number of the derived graph $X'$.
\end{thm}

The above estimate will be referred to as the \emph{relative bound} for the independence number. The usability of the relative bound depends, of course, on being able to give an upper estimate for the independence number of the derived graph. The underlying principle is that the derived graph is often a simpler and much smaller graph than the original one. If needed, the derived graph could be derived once again, and so the relative bound can be interpreted as a recursive, hierarchical procedure.

An obvious upper estimate for $\alpha'$ is the number of vertices of non-maximal degree, that is, the size of the derived graph $X'$. In many examples of interest, the vertices of non-maximal degree are actually independent in $X$, so this is the most that the relative bound can give. Besides, using the number of vertices of non-maximal degree to estimate $\alpha'$ is often good enough, if that number is sufficiently small.

The other obvious upper estimate for $\alpha'$ is $\alpha$ itself, and then one easily checks that the relative bound amounts to the Hoffman-type bound. This is of theoretical interest, for it shows that Theorem~\ref{thm: rel} refines Theorem~\ref{thm: Hlap}. In fact, the gain in using the relative bound instead of the Hoffman-type bound can be highlighted by writing the bound of Theorem~\ref{thm: rel} as follows:
\begin{align}\label{eq: gain}
\alpha\leq n\Big(1-\frac{\delta}{\lambda_{\max}}\Big)-\frac{\Delta-\delta}{\lambda_{\max}-\delta}\bigg(n\Big(1-\frac{\delta}{\lambda_{\max}}\Big)-\alpha'\bigg)
\end{align}
So, for an irregular graph, the gain reflects the gap in the bound
\begin{align}\label{eq: prime}
\alpha'\leq n\Big(1-\frac{\delta}{\lambda_{\max}}\Big).
\end{align}
If the bound \eqref{eq: prime} is strict, then our relative bound is strictly better than the Hoffman-type bound. If equality holds in \eqref{eq: prime}, then the relative bound and the Hoffman-type bound become equalities as well. A simple example of a graph, with the property that equality is achieved in \eqref{eq: prime}, is given by a bi-regular bipartite graph. A more sophisticated example is the unitary polarity graph (\cite[p.299]{vDH}, \cite[Thm.7]{MW}).

\section{A brief reminder on the Laplacian} 
Let $X$ be a finite simple graph. The Laplacian is a linear operator on the space of complex-valued functions defined on the vertex set $V$ of $X$. This is a finite-dimensional space, endowed with the inner product 
\begin{align*}
\la \phi,\psi\ra=\sum_{v\in V} \phi(v)\overline{\psi(v)}.
\end{align*}
By definition, the Laplacian acts on a complex-valued function $\phi$ defined on $V$, as follows:
\begin{align*}
(L\phi)(v)=\deg(v)\: \phi(v)-\sum_{w:\: w\sim v} \phi(w)
\end{align*}
Here, $\deg(v)$ is the degree of the vertex $v$, and the sum is taken over all neighbours of $v$. The largest Laplacian eigenvalue $\lambda_{\max}$ satisfies
\begin{align}\label{eq: lapeigen}
\la L\phi, \phi\ra\leq \lambda_{\max} \la \phi, \phi\ra,
\end{align}
and the left-hand side of \eqref{eq: lapeigen} admits a very useful alternate formula:
\begin{align}\label{eq: lapuse}
\la L\phi, \phi\ra=\sum_{\{v,w\}\in E} |\phi(v)-\phi(w)|^2
\end{align}
where the sum is taken over all the edges of $X$. An absolute difference $|\phi(v)-\phi(w)|$ will be referred to as the edge differential of $\phi$ over the edge $\{v,w\}$.


\section{Discussion of the relative bound}
\subsection{Proof of Theorem~\ref{thm: rel}} Let $U$ be a non-empty independent set of vertices in $X$. Partition $U$ into two subsets, according to their vertex degree in $X$: let $U_1$ contain the vertices of $U$ having non-maximal degree, and let $U_2$ contain the vertices of maximal degree. Put $n_1=|U_1|$ and $n_2=|U_2|$, so $|U|=n_1+n_2$. Note that $n_1\leq \alpha'$, as $U_1$ is an independent set of vertices in the derived graph $X'$.  

Define a function on the vertices of $X$ as follows:
\begin{align*}
f=a_1\cdot \ct_{U_1}+a_2 \cdot \ct_{U_2}-c\cdot \ct=
\begin{cases}
a_1-c & \textrm{ on } U_1\\a_2-c & \textrm{ on } U_2\\-c & \textrm{ on } U^c
\end{cases}
\end{align*}
where $U^c$ denotes the complement of $U$. The real numbers $a_1$, $a_2$, and $c$ are subject to the requirement that $f$ be orthogonal to the constant function $\ct$. As
\begin{align*}
\la f, \ct\ra= a_1\la \ct_{U_1}, \ct\ra+a_2\la \ct_{U_2}, \ct\ra-c\la\ct, \ct\ra =n_1a_1+n_2a_2-nc,
\end{align*}
this means that $c$ is determined by the relation $nc=n_1a_1+n_2a_2$.

We bring in the largest Laplacian eigenvalue $\lambda:=\lambda_{\max}$ by means of the inequality \eqref{eq: lapeigen}, $\la Lf, f\ra\leq \lambda \la f, f\ra$. Firstly, we compute
\begin{align*}
\la f, f\ra&= \la f, a_1\cdot \ct_{U_1}\ra+\la f, a_2\cdot \ct_{U_2}\ra=a_1(a_1-c)n_1+a_2(a_2-c)n_2\\
&=n_1a_1^2+n_2a_2^2-(n_1a_1+n_2a_2)c=n_1a_1^2+n_2a_2^2-\frac{1}{n}(n_1a_1+n_2a_2)^2.
\end{align*}
Secondly, we estimate the term $\la Lf, f\ra$, and we do so by relying on the formula \eqref{eq: lapuse}. The independence of $U$ means that edges in $X$ either join $U$ to $U^c$, or they are internal to $U^c$; the latter ones have, however, vanishing edge differential. Edges joining $U$ to $U^c$ split as follows: there are $\Delta n_2$ edges between $U_2$ and $U^c$, each with an edge differential equal to $|a_2|$, and there are at least $\delta n_1$ edges between $U_1$ and $U^c$, each with an edge differential equal to $|a_1|$. Thus
\begin{align*}
\la Lf, f\ra\geq \delta n_1a_1^2+\Delta n_2a_2^2.
\end{align*}
Plugging in the above estimates into the inequality $\la Lf, f\ra\leq \lambda \la f, f\ra$, and rewriting, brings us to the following:
\begin{align*}
n_1\big((\lambda-\delta)n-\lambda n_1\big)a_1^2+n_2\big((\lambda -\Delta)n-\lambda n_2\big)a_2^2-2\lambda n_1n_2a_1a_2\geq 0 \tag{$*$}
\end{align*}
Viewing the left-hand side of $(*)$ as a quadratic form in $a_1$ and $a_2$, we infer that its discriminant is non-positive. Dividing through by $4n_1n_2$, this says that
\begin{align*}
\lambda^2 n_1n_2\leq \big((\lambda-\delta)n-\lambda n_1\big)\big((\lambda -\Delta)n-\lambda n_2\big)
\end{align*}
and so, after cancelling the term $\lambda^2 n_1n_2$ and dividing through by $n$, we get
\begin{align*}
\lambda(\lambda-\Delta)n_1+\lambda(\lambda-\delta)n_2\leq (\lambda-\Delta)(\lambda-\delta)n.\tag{$**$}
\end{align*}
On the way, we have assumed that $n_1$ and $n_2$ are non-zero. Note that, if $n_1=0$ or $n_2=0$, it is still true that $(*)$ implies $(**)$. 

Replacing $n_2=|U|-n_1$ on the left-hand side of  $(**)$, we can rewrite it as
\begin{align*}
\lambda(\lambda-\delta)|U|\leq (\lambda-\Delta)(\lambda-\delta)n+\lambda(\Delta-\delta)n_1.
\end{align*}
Finally, we bound $n_1\leq \alpha'$, we divide through by $\lambda(\lambda-\delta)$, and we finally get
\begin{align*}
|U|\leq n\Big(1-\frac{\Delta}{\lambda}\Big)+\alpha'\frac{\Delta-\delta}{\lambda-\delta},
\end{align*}
which completes the proof.

\subsection{A weaker bound, after Godsil and Newman} In order to illuminate the above proof and its outcome, let us consider the following argument. 

Let $U$ be a non-empty set of vertices in $X$, and let $g$ be the map on the vertices of $X$ defined by $g= n-|U|$ on $U$, and $g= -|U|$ on $U^c$. Then
\begin{align*}
\la g,g\ra=n|U|(n-|U|), \qquad \la Lg,g\ra=n^2 e(U,U^c)
\end{align*}
where $e(U,U^c)$ denotes the number of edges joining vertices in $U$ to vertices in $U^c$. If $U$ is an independent set, then
\begin{align*}
e(U,U^c)=\sum_{v\in U} \mathrm{deg}(v)= |U| \:\deg(U)
\end{align*}
where $\deg(U)=|U|^{-1}\sum_{v\in U}\mathrm{deg}(v)$ denotes the average degree over $U$. Plugging in these computations into the inequality $\la Lg, g\ra\leq \lambda \la g, g\ra$, where $\lambda:=\lambda_{\max}$, leads to
\begin{align}
|U|\leq n\Big(1-\frac{\deg(U)}{\lambda}\Big) \label{IH}.
\end{align}
The bound \eqref{IH} is due to Godsil and Newman \cite[Cor.3.5]{GN}. Obviously, it implies the Hoffman-type bound, as $\deg(U)\geq \delta$. 

Following \cite{GN}, we can derive from \eqref{IH} an explicit bound for the independence number of $X$. As before, partition $U$ into $n_1$ vertices of non-maximal degree, and $n_2$ vertices having maximal degree in $X$. Note that $n_1$ is at most $\alpha'$, the independence number of the derived graph $X'$. The average degree over $U$ can then be lower-bounded as follows:
\begin{align*}
\deg(U)\geq \frac{1}{|U|}(n_1\delta+n_2\Delta)= \Delta-n_1\:\frac{\Delta-\delta}{|U|}\geq \Delta-\alpha'\frac{\Delta-\delta}{|U|}.
\end{align*}
Combining this lower bound with \eqref{IH}, one gets a quadratic inequality for $|U|$, whence an explicit upper bound for $|U|$. The conclusion is that the independence number of $X$ satisfies
\begin{align}\label{GN-type}
\alpha \leq\frac{n}{2}\Bigg(1-\frac{\Delta}{\lambda}+\sqrt{\Big(1-\frac{\Delta}{\lambda}\Big)^2+\frac{4\alpha'(\Delta-\delta)}{n\lambda}}\Bigg).
\end{align}
The bound \eqref{GN-type} is weaker than the relative bound, but stronger than the Hoffman-type bound. This can be worked out directly, and both assertions reduce, after calculations, to the fact that \eqref{eq: prime} holds in an irregular graph. If \eqref{eq: prime} is strict, then the ordering of the three bounds is strict as well. If equality holds in \eqref{eq: prime}, then the three bounds agree.

Godsil and Newman \cite{GN} actually work out a particular instance of the bound \eqref{GN-type}, and \eqref{GN-type} can be viewed as a direct descendant of arguments from \cite{GN}. The formal novelty is the consideration of the derived graph $X'$. Our relative bound, on the other hand, arises from a new idea: that of optimally weighting the edge-count given by a splitting of an independent set along maximal/non-maximal degrees. 

\subsection{Monotony} The spectral bounding function which appears in Theorem~\ref{thm: Hlap}, namely
\begin{align*}
\lambda\mapsto n\Big(1-\frac{\delta}{\lambda}\Big),
\end{align*}
is obviously increasing. We record the following elementary lemma, which shows that the spectral bounding function which appears in Theorem~\ref{thm: rel} is also increasing in a relevant range. This increasing behaviour is very useful: often, we only know an upper bound for the largest Laplacian eigenvalue $\lambda_{\max}$, rather than an exact value. 

\begin{lem}\label{lem: monoton} The function
\begin{align*}
\lambda\mapsto n\Big(1-\frac{\Delta}{\lambda}\Big)+\alpha'\frac{\Delta-\delta}{\lambda-\delta}
\end{align*}
is increasing for $\lambda\geq \lambda_{\max}$.
\end{lem}

\begin{proof} If $b(\lambda)$ denotes the given function, then the condition $b'(\lambda)>0$ can be written as
\begin{align*}
\frac{\alpha'}{n}\Big(1-\frac{\delta}{\Delta}\Big)<\Big(1-\frac{\delta}{\lambda}\Big)^2.
\end{align*} 
We need to check that the above inequality holds for $\lambda\geq \lambda_{\max}$. Indeed, we have
\begin{align*}
\frac{\alpha'}{n}\leq 1-\frac{\delta}{\lambda_{\max}}, \qquad 1-\frac{\delta}{\Delta}<1-\frac{\delta}{\lambda_{\max}},
\end{align*}
by the Hoffman-type bound, respectively thanks to the fact that $\lambda_{\max}\geq \Delta+1>\Delta$. It follows that
\begin{align*}
\frac{\alpha'}{n}\Big(1-\frac{\delta}{\Delta}\Big)< \Big(1-\frac{\delta}{\lambda_{\max}}\Big)^2\leq \Big(1-\frac{\delta}{\lambda}\Big)^2
\end{align*}
as desired.
\end{proof}
\section{Examples, part I}
In the first two examples, we test the relative bound on graphs whose independence number is actually known. We will see that it performs much better than the Hoffman-type bound. In fact, the relative bound turns out to be integrally sharp, in the sense that the integral part of the upper bound equals the independence number. 

The third example discusses the relative bound in the context of cartesian products.

\subsection{Path graphs} The path graph $P_n$ on $n$ vertices, where $n\geq 3$, has independence number $\alpha(P_n)=\lceil n/2\rceil$. Furthermore, note that $P_n$ has maximal degree $\Delta=2$, minimal degree $\delta=1$, and largest Laplacian eigenvalue $\lambda_{\max}=2+2\cos(\pi/n)=4-\Theta(n^{-2})<4$. 

The Hoffman-type bound gives
\begin{align*}
\alpha(P_n)\leq n\Big(1-\frac{1}{\lambda_{\max}}\Big)<\frac{3n}{4}.
\end{align*}
Now let us apply the relative bound. The derived graph of $P_n$ consists of two disconnected nodes, so $\alpha'=2$. Therefore
\begin{align*}
\alpha(P_n)\leq n\Big(1-\frac{2}{\lambda_{\max}}\Big)+\frac{2}{\lambda_{\max}-1}<\frac{n}{2}+\frac{2}{3},
\end{align*}
the latter inequality owing to monotony (Lemma~\ref{lem: monoton}). The integral part of the right-hand side is $\lceil n/2\rceil$, so the above bound is integrally sharp.

\subsection{Cones} Let $X$ be any non-empty graph on $n$ vertices, except for the complete graph $K_n$. Consider the cone $\widehat{X}$ over $X$, obtained by adding a brand new vertex and then joining it to every vertex of $X$. The cone $\widehat{X}$ has $n+1$ vertices, and independence number $\alpha(\widehat{X})=\alpha$, the independence number of $X$. Furthermore, $\widehat{X}$ has maximal degree $n$, minimal degree $\delta+1$, where $\delta$ is the minimal degree of $X$, and $\lambda_{\max}=n+1$. 

The Hoffman-type bound gives
\begin{align*}
\alpha(\widehat{X})\leq n-\delta.
\end{align*}
This reads, in effect, as the basic upper bound $\alpha\leq n-\delta$ for $X$, a bound which is often very weak.

In order to apply the relative bound, we start by noting that the derived graph of $\widehat{X}$ is the base graph $X$. The relative bound then gives
\begin{align*}
\alpha(\widehat{X})\leq 1+\alpha-\frac{\alpha}{n-\delta}
\end{align*}
which is integrally sharp, as the right-hand side has integral part $\alpha$.

By way of contrast, the bound \eqref{GN-type} gives
\begin{align*}
\alpha(\widehat{X})\leq \frac{1}{2}\big(1+\sqrt{1+4\alpha(n-\delta-1)}\big)
\end{align*}
which, just like the Hoffman-type bound, might be far from the correct value $\alpha(\widehat{X})=\alpha$.

\subsection{Cartesian products} Let $X$ and $Y$ be non-empty graphs. The size, maximal degree, minimal degree, and the largest laplacian eigenvalue of $X$ and $Y$ are denoted $n_X,\Delta_X,\delta_X,\lambda_X$, respectively $n_Y,\Delta_Y,\delta_Y,\lambda_Y$.

Consider the cartesian product graph $X\square Y$. Recall, this graph has vertex set $X\times Y$, and two vertices $(x_1,y_1)$ and $(x_2,y_2)$ are adjacent whenever $x_1=x_2$ and $y_1$ is adjacent to $y_2$ in $Y$, or $x_1$ is adjacent to $x_2$ in $X$ and $y_1=y_2$. The size, maximal degree, minimal degree, and the largest laplacian eigenvalue of $X\square Y$ are $n_Xn_Y,\Delta_X+\Delta_Y,\delta_X+\delta_Y,\lambda_X+\lambda_Y$.

The basic upper bound for the independence number of $X\square Y$ is
\begin{align}\label{eq: viz}
\alpha(X\square Y)\leq \min\big\{\alpha(X)\: n_Y, \alpha(Y)\: n_X\big\}.
\end{align}
The Hoffman-type bound gives
\begin{align}
\alpha(X\square Y)&\leq n_Xn_Y\min\bigg\{1-\frac{\delta_X}{\lambda_X}, 1-\frac{\delta_Y}{\lambda_Y}  \bigg\}\label{eq: hofone}\\
&\leq n_Xn_Y\bigg(1-\frac{\delta_X+\delta_Y}{\lambda_X+\lambda_Y}\bigg)\label{eq: hoftwo}.
\end{align}
The bound \eqref{eq: hofone} combines \eqref{eq: viz} and the Hoffman-type bound for each factor, whereas \eqref{eq: hoftwo} is a direct application of the Hoffman-type bound for the product. Clearly, \eqref{eq: hofone} is a better bound than \eqref{eq: hoftwo}.

The relative bound says that the independence number of $X\square Y$ satisfies
\begin{align}\label{eq: relprod}
\alpha(X\square Y)\leq n_Xn_Y\Big(1-\frac{\Delta_X+\Delta_Y}{\lambda_X+\lambda_Y}\Big)+\alpha((X\square Y)')\frac{(\Delta_X+\Delta_Y)-(\delta_X+\delta_Y)}{(\lambda_X+\lambda_Y)-(\delta_X+\delta_Y)}.
\end{align}
The vertices of maximal degree in $X\square Y$ are pairs of vertices of maximal degree in $X$, respectively $Y$. So the derived graph $(X\square Y)'$ is the union of the two subgraphs $X'\square Y$ and $X\square Y'$. It follows that the independence number of $(X\square Y)'$ obeys the Leibniz-like rule
\begin{align*}
\alpha((X\square Y)')\leq \alpha(X'\square Y)+\alpha(X\square Y').
\end{align*}
Let us give a concrete example, in which the relative bound \eqref{eq: relprod} beats the basic bound \eqref{eq: viz}. We consider a cartesian product of the form $X\square P_{2k}$, where $X$ is a complete split graph on $n$ vertices, and $P_{2k}$ is the path graph on $2k$ vertices. Specifically, the graph $X$ is the join $K_{(1-\e)n}\vee (\e n)K_1$ of the complete graph on $(1-\e)n$ nodes with the empty graph on $\e n$ nodes. We assume that $k>1$ and $\e n>1$. 

As $P_{2k}$ has independence number $\alpha(P_{2k})=k$, and $X$ has independence number $\alpha(X)=\e n$, the basic bound \eqref{eq: viz} becomes 
\begin{align}\label{eq: bd1}
\alpha(X\square P_{2k})\leq \min\{2\e kn, kn\}.
\end{align}
On the other hand, the relative bound \eqref{eq: relprod} leads to 
\begin{align}\label{eq: bd2}
\alpha(X\square P_{2k})\leq \e k n+2\e n+3k.
\end{align}
We forgo the tedious details, but we highlight the key points. Firstly, the graph $X$ has maximal degree $\Delta=n-1$, minimal degree $\delta=(1-\e)n$, and largest laplacian eigenvalue $\lambda_{\max}=n$. Secondly, the largest Laplacian eigenvalue of $P_{2k}$ is less than $4$, and one proceeds by using Lemma~\ref{lem: monoton}. Thirdly, the derived graph of $X$ consists of $\e n$ independent vertices. So the independence number of the derived graph of $X\square P_{2k}$ can be bounded as follows:
\begin{align*}
\alpha((X\square P_{2k})')\leq \alpha(X'\square P_{2k})+\alpha(X\square P_{2k}')=k(\e n)+2(\e n).
\end{align*}
Now, let us think of $n$ and $k$ as being large, and $\e$ as being fixed in $(0,\tfrac{1}{2})$. Then \eqref{eq: bd2} improves \eqref{eq: bd1}, essentially by a factor of $2$. In fact, \eqref{eq: bd2} gives the correct order of magnitude: the Vizing lower bound
\begin{align*}
\alpha(X\square Y)\geq \alpha(X)\alpha(Y)+\min\big\{n_X-\alpha(X),n_Y- \alpha(Y)\big\}
\end{align*}
implies that $\alpha(X\square P_{2k})\geq \e kn+\min\{k,(1-\e)n\}$.

\section{Examples, part II}
The graphs considered in this section will be \emph{nearly regular}. A graph $X$ is said to be nearly $d$-regular if the vertex degrees are $d$ or $d-1$. The vertices of degree $d-1$ are thought of as being deficient, and by the \emph{deficiency} of $X$ we mean the number of deficient vertices. The derived graph $X'$ is the subgraph induced by the deficient vertices. 

Note that, for nearly regular graphs, we expect the gain in using the relative bound over the Hoffman-type bound to be relatively small.

\subsection{The Erd\H{o}s - R\'enyi graph} Let $\F$ be a finite field with $q$ elements. The \emph{Erd\H{o}s - R\'enyi graph} over $\F$ has the projective plane $(\F^3)^*/\F^*$ as its vertex set, and two distinct vertices $[x_1,x_2,x_3]$ and $[y_1,y_2,y_3]$ are joined by an edge whenever $x_1y_1+x_2y_2+x_3y_3=0$. This graph, denoted $ER_q$ in what follows, has $q^2+q+1$ vertices, it is nearly $(q+1)$-regular, and it has deficiency $q+1$. Furthermore, the deficient vertices are independent.

The Erd\H{o}s - R\'enyi graph was introduced in \cite{ER} and, independently, in \cite{Br}, as a $C_4$-free graph with many edges. It is an early, and distinguished, example in Tur\'an-type extremal graph theory. Recently, the independence number of the Erd\H{o}s - R\'enyi graph has been the subject of some attention. Let us give a quick overview. 

The non-trivial Laplacian eigenvalues of the Erd\H{o}s - R\'enyi graph are $q+1\pm \sqrt{q}$, and so $\lambda_{\max}=q+1+ \sqrt{q}$. The Hoffman-type bound gives 
\begin{align*}
\alpha(ER_q)\leq q\sqrt{q}+1.
\end{align*} 
Godsil and Newman \cite{GN}, using an instance of the bound \eqref{GN-type}, showed that
\begin{align}\label{eq: ER}
\alpha(ER_q)<q\sqrt{q}-q+2\sqrt{q}.
\end{align}

Refinements of \eqref{eq: ER} were pursued in \cite{KN+} and \cite{HW}. In \cite{HW}, it is shown that, in the case when $q$ is even, $\alpha(ER_q)<q\sqrt{q}-q+\sqrt{q}+2$.

Lower bounds for the independence number of the Erd\H{o}s - R\'enyi graph are given in \cite{MW}, with some recent improvements in \cite{MPS}. The overall feature is that $\alpha(ER_q)\geq Cq\sqrt{q}$ for some explicit numerical constant $C>0$. Notably, if $q$ is an even power of $2$, then $\alpha(ER_q)\geq q\sqrt{q}-q+\sqrt{q}$ \cite[Thm.6]{MW}.

Now, we can apply our relative bound, using $\alpha'=q+1$. The outcome is strictly better than \eqref{eq: ER}, but only marginally so. A more interesting application of the relative bound comes up in the following generalization.

\subsection{Orthogonality graphs} The Erd\H{o}s - R\'enyi graph is just the first in a family of graphs defined by orthogonality. Again, let $\F$ be a finite field with $q$ elements. We consider the usual inner product on $\F^n$, $n\geq 3$, given by $x\cdot y=x_1y_1+\dots+x_ny_n$. The \emph{orthogonality graph}, denoted $O^n_q$, has the projective plane $(\F^n)^*/\F^*$ as its vertex set, and two distinct vertices $[x]$ and $[y]$ are joined by an edge whenever $x\cdot y=0$. So the independence number of $O^n_q$ is interpreted, geometrically, as the largest number of lines through the origin in $\F^n$, no two of which are orthogonal. 

In what follows, it will be convenient to use the notation $[k]_q=\frac{q^k-1}{q-1}$.

The orthogonality graph $O^n_q$ has $[n]_q$ vertices, and it is nearly $[n-1]_q$-regular. By \cite[Sec.5]{vDH}, the non-trivial Laplacian eigenvalues of $O^n_q$ are
\begin{align*}
[n-1]_q\pm \sqrt{[n-1]_q-[n-2]_q}=[n-1]_q\pm q^{n/2-1}
\end{align*}
and so $\lambda_{\max}=[n-1]_q+q^{n/2-1}$. The Hoffman-type bound gives, after a pleasant computation, the following estimate on the independence number of the orthogonality graph:
\begin{align*}
\alpha(O^n_q)\leq q^{n/2}+1.
\end{align*}
Now let us work out the relative bound. Having computed the Hoffman-type bound, it is convenient to resort to \eqref{eq: gain}; we get 
\begin{align}\label{eq: lazy}
\alpha(O^n_q)\leq q^{n/2}+1-\frac{q^{n/2}+1-\alpha'}{q^{n/2-1}+1}=q^{n/2}+1-q+\frac{\alpha'+q-1}{q^{n/2-1}+1}.
\end{align}
We need to handle the main technical ingredient: the independence number, $\alpha'$, of the derived graph of $O^n_q$. It turns out that the derived graph is highly symmetric, namely a strongly regular graph, in most cases. This has been thoroughly studied by Parsons, and we will use several salient results from \cite{P}.

\subsubsection{Odd dimension} Assume $n>4$ is odd. Then the derived graph of $O^n_q$ is a connected strongly regular graph with parameters 
\begin{align*}
([n-1]_q, [n-2]_q-1, [n-3]_q-2, [n-3]_q).
\end{align*}
See \cite[Thm.2(vi)]{P} for odd $q$, and \cite[Thm.5.A(i); Thm.4]{P} for even $q$. The Hoffman bound gives, after a less pleasant computation, the estimate
\begin{align*}
\alpha'\leq q^{(n-1)/2}+1.
\end{align*}
Using \eqref{eq: lazy}, we deduce that
\begin{align*}
\alpha(O^n_q)\leq q^{n/2}-q+\sqrt{q}+2.
\end{align*}

\subsubsection{Even dimension} Assume $n>4$ is even, and $q$ is odd. Then the derived graph of $O^n_q$ is a connected strongly regular graph with parameters 
\begin{align*}
([n-1]_q+\e q^{n/2-1}, [n-2]_q-1+\e q^{n/2-1}, [n-3]_q-2+\e q^{n/2-1}, [n-3]_q+\e q^{n/2-2})
\end{align*}
where $\e$ is a signing defined by $\e=\sigma(-1)^{n/2}$, $\sigma$ being the quadratic character on $\F$. See \cite[Thm.3(i); (3) in Sec.7]{P}. There is a good structural description in the case when $q$ is even, as well \cite[Thm.5.B(i); (6) in Sec.7]{P}: the derived graph of $O^n_q$ is a cone over a regular graph that is not too far from being strongly regular. For the sake of simplicity, we disregard this case. 

The Hoffman bound gives, after a tedious computation, the following estimate:
\begin{align*}
\alpha'\leq 
\begin{cases} q^{n/2-1}+1 & \textrm{ if } \e=1\\
 q^{n/2}+1 & \textrm{ if } \e=-1
 \end{cases}
\end{align*}
If $\e=-1$, this analysis brings no improvement to the Hoffman-type bound. The satisfactory case is $\e=1$; this happens if and only if $n\equiv 0$ mod $4$, or $q\equiv 1$ mod $4$. Using \eqref{eq: lazy}, we deduce that
\begin{align*}
\alpha(O^n_q)\leq q^{n/2}-q+2.
\end{align*}


\end{document}